\documentclass[a4paper]{amsart}

\usepackage{tikz,mathrsfs}
\usepackage{graphicx}
\usepackage{amsmath}
\usepackage{amssymb}
\usepackage{oldgerm}
\usepackage[all]{xy}

\newtheorem{theorem}{Theorem}[section]
\newtheorem{lemma}[theorem]{Lemma}
\newtheorem{proposition}[theorem]{Proposition}
\newtheorem{corollary}[theorem]{Corollary}
\theoremstyle{definition}
\newtheorem*{definition}{Definition}
\newtheorem*{question}{Question}

\newtheorem{example}[theorem]{Example}

\DeclareMathOperator{\Image}{Im}

\DeclareMathOperator{\Hom}{Hom}
\DeclareMathOperator{\Ext}{Ext}

\DeclareMathOperator{\modules}{mod} \renewcommand{\mod}{\modules}

\DeclareMathOperator{\proj}{proj}

\DeclareMathOperator{\rank}{rank}
\DeclareMathOperator{\s}{\Sigma}
\DeclareMathOperator{\thick}{thick}

\DeclareMathOperator{\cx}{cx}
\DeclareMathOperator{\Ho}{H}

\newcommand{\K}{\mathbf{K}}
\newcommand{\D}{\mathbf{D}}
\newcommand{\Kac}{\K_{\rm tac}}

\usetikzlibrary{arrows}

\tikzset{>=stealth'}

\def\arrowLengthDisplayStyle{4ex}
\def\arrowHeightDisplayStyle{.8ex}
\def\arrowSkipDisplayStyle{.5ex}
\def\arrowLengthTextStyle{3ex}
\def\arrowHeightTextStyle{.8ex}
\def\arrowSkipTextStyle{.4ex}
\def\arrowLengthScriptStyle{2.5ex}
\def\arrowHeightScriptStyle{.6ex}
\def\arrowSkipScriptStyle{.3ex}
\def\arrowLengthScriptScriptStyle{2ex}
\def\arrowHeightScriptScriptStyle{.4ex}
\def\arrowSkipScriptScriptStyle{.2ex}

\renewcommand{\to}{\arrow{->}}

\newcommand{\MakeTikzArrowWithSuperscriptSubscript}[4]
 {
  \mathchoice
   { %Display style
    \hspace*{\arrowSkipDisplayStyle}
    \begin{tikzpicture}[baseline]
     \draw [#1] (0,\arrowHeightDisplayStyle) -- node [above] {$#2$} node [below] {$#3$} (#4 * \arrowLengthDisplayStyle, \arrowHeightDisplayStyle);
    \end{tikzpicture}
    \hspace*{\arrowSkipDisplayStyle}
   }
   { %Text style
    \hspace*{\arrowSkipTextStyle}
    \begin{tikzpicture}[baseline]
     \draw [#1] (0,\arrowHeightTextStyle) -- node [above] {$\scriptstyle #2$} node [below] {$\scriptstyle #3$} (#4 * \arrowLengthTextStyle, \arrowHeightTextStyle);
    \end{tikzpicture}
    \hspace*{\arrowSkipTextStyle}
   }
   { %Script style
    \hspace*{\arrowSkipScriptStyle}
    \begin{tikzpicture}[baseline]
     \draw [#1] (0,\arrowHeightScriptStyle) -- node [above] {$\scriptscriptstyle #2$} node [below] {$\scriptscriptstyle #3$} (#4 * \arrowLengthScriptStyle, \arrowHeightScriptStyle);
    \end{tikzpicture}
    \hspace*{\arrowSkipScriptStyle}
   }
   { %Script script style
    \hspace*{\arrowSkipScriptScriptStyle}
    \begin{tikzpicture}[baseline]
     \draw [#1] (0,\arrowHeightScriptScriptStyle) -- node [above] {$\scriptscriptstyle #2$} node [below] {$\scriptscriptstyle #3$} (#4 * \arrowLengthScriptScriptStyle, \arrowHeightScriptScriptStyle);
    \end{tikzpicture}
    \hspace*{\arrowSkipScriptScriptStyle}
   }
 }

\newcommand{\MakeTikzArrowWithCentralLabel}[3]
 {
  \mathchoice
   { %Display style
    \hspace*{\arrowSkipDisplayStyle}
    \begin{tikzpicture}[baseline]
     \draw [#1] (0,\arrowHeightDisplayStyle) -- node [fill=white,inner sep=1pt] {$#2$} (#3 * \arrowLengthDisplayStyle, \arrowHeightDisplayStyle);
    \end{tikzpicture}
    \hspace*{\arrowSkipDisplayStyle}
   }
   { %Text style
    \hspace*{\arrowSkipTextStyle}
    \begin{tikzpicture}[baseline]
     \draw [#1] (0,\arrowHeightTextStyle) -- node [fill=white,inner sep=1pt] {$\scriptstyle #2$} (#3 * \arrowLengthTextStyle, \arrowHeightTextStyle);
    \end{tikzpicture}
    \hspace*{\arrowSkipTextStyle}
   }
   { %Script style
    \hspace*{\arrowSkipScriptStyle}
    \begin{tikzpicture}[baseline]
     \draw [#1] (0,\arrowHeightScriptStyle) -- node [fill=white,inner sep=1pt] {$\scriptscriptstyle #2$} (#3 * \arrowLengthScriptStyle, \arrowHeightScriptStyle);
    \end{tikzpicture}
    \hspace*{\arrowSkipScriptStyle}
   }
   { %Script script style
    \hspace*{\arrowSkipScriptScriptStyle}
    \begin{tikzpicture}[baseline]
     \draw [#1] (0,\arrowHeightScriptScriptStyle) -- node [fill=white,inner sep=1pt] {$\scriptscriptstyle #2$} (#3 * \arrowLengthScriptScriptStyle, \arrowHeightScriptScriptStyle);
    \end{tikzpicture}
    \hspace*{\arrowSkipScriptScriptStyle}
   }
 }

\def\arrow#1{\def\lastArrowStyle{#1}
             \futurelet\testchar\arrowMaybeStreched}
\def\arrowMaybeStreched{\ifx[\testchar \let\next\arrowStreched
                         \else \let\next\arrowUnstreched \fi
                        \next}

\def\arrowStreched[#1]{\def\lastArrowStrech{#1}
                       \futurelet\testchar\arrowMaybeLabel}
\def\arrowUnstreched{\def\lastArrowStrech{1}
                     \futurelet\testchar\arrowMaybeLabel}

\def\arrowMaybeLabel{\ifx^\testchar \let\next\arrowSuperscript
                      \else \ifx_\testchar \let\next\arrowSubscript
                             \else \ifx~\testchar \let\next\arrowCentralLabel
                                    \else \let\next\arrowNoLabel
                                   \fi
                            \fi
                     \fi
                     \next}

\def\arrowSuperscript^#1{\def\lastArrowSuperscript{#1}
                         \futurelet\testchar\arrowSuperMaybeSub}
\def\arrowSuperMaybeSub{\ifx_\testchar \let\next\arrowSuperscriptSubscript
                         \else \let\next\arrowSuperscriptNoSubscript \fi
                        \next}

\def\arrowSubscript_#1{\def\lastArrowSubscript{#1}
                         \futurelet\testchar\arrowSubMaybeSuper}
\def\arrowSubMaybeSuper{\ifx^\testchar \let\next\arrowSubscriptSuperscript
                         \else \let\next\arrowSubscriptNoSuperscript \fi
                        \next}

\def\arrowSuperscriptSubscript_#1{\def\lastArrowSubscript{#1}
                                  \arrowDrawSupSub}
\def\arrowSuperscriptNoSubscript{\def\lastArrowSubscript{}
                                 \arrowDrawSupSub}
\def\arrowSubscriptSuperscript^#1{\def\lastArrowSuperscript{#1}
                                  \arrowDrawSupSub}
\def\arrowSubscriptNoSuperscript{\def\lastArrowSuperscript{}
                                 \arrowDrawSupSub}
\def\arrowNoLabel{\def\lastArrowSuperscript{}
                  \def\lastArrowSubscript{}
                  \arrowDrawSupSub}

\def\arrowCentralLabel~#1{\MakeTikzArrowWithCentralLabel{\lastArrowStyle}{#1}{\lastArrowStrech}}
\def\arrowDrawSupSub{\MakeTikzArrowWithSuperscriptSubscript{\lastArrowStyle}{\lastArrowSuperscript}{\lastArrowSubscript}{\lastArrowStrech}}

\begin{document}

\title{The Gorenstein defect category}

\author{Petter Andreas Bergh, David A.\ Jorgensen \& Steffen Oppermann}

\address{Petter Andreas Bergh \\ Institutt for matematiske fag \\
  NTNU \\ N-7491 Trondheim \\ Norway}
\email{bergh@math.ntnu.no}

\address{David A.\ Jorgensen \\ Department of mathematics \\ University
of Texas at Arlington \\ Arlington \\ TX 76019 \\ USA}
\email{djorgens@uta.edu}

\address{Steffen\ Oppermann \\ Institutt for matematiske fag \\
  NTNU \\ N-7491 Trondheim \\ Norway}
\email{steffen.oppermann@math.ntnu.no}



\begin{abstract}
We consider the homotopy category of complexes of projective modules over a Noetherian ring. Truncation at degree zero induces a fully faithful triangle functor from the totally acyclic complexes to the singularity category. We show that if the ring is either Artin or commutative Noetherian local, then the functor is dense if and only if the ring is Gorenstein. Motivated by this, we define the \emph{Gorenstein defect category} of the ring, a category which in some sense measures how far the ring is from being Gorenstein.
\end{abstract}

\maketitle

\section{Introduction}

Given a ring, one can associate to it certain triangulated categories: derived categories, homotopy categories, and various triangulated subcategories of these, such as bounded derived categories or homotopy categories of acyclic complexes. When the ring is Gorenstein, classical results by Buchweitz (cf.\ \cite{Buchweitz}) show that some of these triangulated categories are equivalent: the stable category of maximal Cohen-Macaulay modules, the singularity category of finitely generated modules, and the homotopy category of totally acyclic complexes of finitely generated projective modules. Thus, for Gorenstein rings, these triangulated categories (together with their respective cohomology theories) virtually coincide.

In this paper, we provide a categorical characterization of Gorenstein rings. Let $A$ be a left Noetherian ring, and $\proj A$ the category of finitely generated projective left $A$-modules. Brutal truncation at degree zero induces a map from the homotopy category $\Kac(\proj A)$ of totally acyclic complexes to the homotopy category $\K^{-,\rm b}(\proj A)$ of right bounded eventually acyclic complexes. However, this map is not a functor. We therefore consider instead the singularity category
$$\D^{\rm b}_{\rm sg}(A) \stackrel{\text{def}}{=} \K^{-,\rm b}(\proj A) / \K^{\rm b}(\proj A)$$
of $A$, where $\K^{\rm b}(\proj A)$ is the homotopy category of bounded complexes. Brutal truncation now induces a triangle functor
$$\beta_{\proj A} \colon \Kac(\proj A) \to \D^{\rm b}_{\rm sg}(A),$$
and we show that this functor is always full and faithful. The main result of section \ref{functor}, Theorem \ref{Gorenstein}, shows that the properties of this functor actually characterize Gorenstein rings. Namely, if $A$ is either an Artin ring or a commutative Noetherian local ring, then the functor $\beta_{\proj A}$ is dense if and only if $A$ is Gorenstein. The ``if" part here is classical: it is part of \cite[Theorem 4.4.1]{Buchweitz}. What is new is the ``only if" part, which may be considered as a triangulated category reformulation of Auslander and Bridger \cite[Theorem (4.20)]{AuslanderBridger}, stating that rings over which every finitely generated module has finite Gorenstein dimension are precisely the Gorenstein rings.  

Motivated by the discussion above, in Section \ref{defect} we define various triangulated defect categories as follows. Let $\mathcal C$ be a thick subcategory of $\Kac(\proj A)$ and 
$\beta_{\mathcal C}$ the restriction of $\beta_{\proj A}$ to $\mathcal C$. We define the \emph{$\mathcal C$ defect category} $\D^{\rm b}_{\mathcal C}(A)$ of $A$ as the Verdier quotient
$$
\D^{\rm b}_{\mathcal C}(A) \stackrel{\text{def}}{=} \D^{\rm b}_{\rm sg}(A) / \langle \Image\beta_{\mathcal C} \rangle,
$$
where $\langle \Image \beta_{\mathcal C} \rangle$ is the isomorphism closure of the image of $\beta_{\mathcal C}$ in $\D^{\rm b}_{\rm sg}(A)$: this is a thick subcategory. Then Theorem \ref{Gorenstein} translates to the following: if $A$ is either an Artin ring or a commutative Noetherian local ring, then $\D^{\rm b}_{\mathcal C}(A) =0$ for $\mathcal C= \Kac(\proj A)$ if and only if $A$ is Gorenstein; in this case we refer to $\D^{\rm b}_{\mathcal C}(A)$ as the \emph{Gorenstein defect category}. When $A$ is commutative local and $\mathcal C$ is the thick subcategory of $\Kac(\proj A)$ consisting of complexes of finite complexity, we show that $\D^{\rm b}_{\mathcal C}(A) =0$ if and only if $A$ is a complete intersection.  The dimension of the Gorenstein defect category, in the sense of Roquier, is therefore in some sense a measure of ``how far" the ring is from being Gorenstein.  

We note that the construction in \cite{IyengarKrause} yields an interesting alternate defect category characterizing Gorenstein rings, namely the Verdier quotient formed by the acyclic complexes of projective modules modulo the totally acyclic complexes.

We end the paper with an example consisting of the `smallest' ring for which the Gorenstein defect category is interesting, and show that the dimension of this category is at most one.

\section{Preliminaries}

Let $\mathscr{P}$ be an additive category, and $\K \mathscr{P}$ the homotopy category of complexes in $\mathscr{P}$. This is a triangulated category, with suspension $\s \colon \K \mathscr{P} \to \K \mathscr{P}$ given by shifting a complex one degree to the left, and changing the sign of its differential. That is, for a complex $C \in \K \mathscr{P}$ with differential $d$, the complex $\s C$ has $C_{n-1}$ in degree $n$, and $-d$ as differential. The (distinguished) triangles in $\K \mathscr{P}$ are the sequences of objects and maps that are isomorphic to sequences of the form
$$
C \xrightarrow{f} C' \to C(f) \to \s C
$$
for some map $f$ and its mapping cone $C(f)$.

We say that a complex $C$ in $\mathscr{P}$ is \emph{acyclic} if the complex $\Hom_{\mathscr{P}}(P, C)$ of abelian groups is acyclic for all objects $P \in \mathscr{P}$. If in addition $\Hom_{\mathscr{P}}(C,P)$ is acyclic for all $P \in \mathscr{P}$, then $C$ is \emph{totally acyclic}. Moreover, $C$ is \emph{eventually acyclic} if for all objects $P \in \mathscr{P}$, the complex $\Hom_{\mathscr{P}}(P, C)$ is eventually acyclic, i.e.\ $\Ho_n \left ( \Hom_{\mathscr{P}}(P, C) \right ) =0$ for $|n| \gg 0$. Note that we cannot define acyclicity directly for complexes in $\mathscr{P}$, since the category is only assumed to be additive.

We shall be working with the following full subcategories of $\K \mathscr{P}$ (the definitions are up to isomorphism in $\K \mathscr{P}$):
\begin{align*} 
\Kac \mathscr{P} & = \{ C \in \K \mathscr{P} \mid C \text{ is totally acyclic} \} \\
\K^{-, \rm b} \mathscr{P} & = \{C \in \K \mathscr{P} \mid C_n =0 \text{ for } n \ll 0
 \text{ and $C$ is eventually acyclic} \} \\
\K^{\rm b} \mathscr{P} & = \{C \in \K \mathscr{P} \mid C_n =0 \text{ for } |n| \gg 0 \}.
\end{align*}
These are all triangulated subcategories of $\K \mathscr{P}$. For example, when $\mathscr{P}$ is the category of finitely generated left projective modules over a left Noetherian ring $A$, then $\K^{\rm b} \mathscr{P}$ is by definition the category of perfect complexes. Moreover, in this setting it is well known that the triangulated categories $\K^{-, \rm b} \mathscr{P}$ and $\D^{\rm b}(\mod A)$ are equivalent, where $\mod A$ is the category of finitely generated left $A$-modules.

The category $\K^{\rm b} \mathscr{P}$ is a thick subcategory of $\K^{-, \rm b} \mathscr{P}$, that is, a triangulated subcategory closed under direct summands. The Verdier quotient $\K^{-, \rm b} \mathscr{P} / \K^{\rm b} \mathscr{P}$ is therefore a well defined triangulated category. Recall that the objects in this quotient are the same as the objects in $\K^{-, \rm b} \mathscr{P}$. A morphism $C \to C'$ in the quotient is an equivalence class of diagrams of the form
\begin{center}
  \begin{tikzpicture}[text centered]
    % Nodes
    \node (C1) at (0,0){$C$};
    \node (C2) at (2,0){$C'$};
    \node (D) at (1,1){$D$};
    
    \begin{scope}[->,font=\scriptsize,midway]
      % Arrows
      \draw (D) -- node[above left=-2pt]{$g$} (C1);
      \draw (D) -- node[above right=-2pt]{$f$} (C2);
      
    \end{scope}
  \end{tikzpicture}
\end{center}
where $f$ and $g$ are morphisms in $\K^{-, \rm b} \mathscr{P}$, and $g$ has the property that its mapping cone $C(g)$ belongs to $\K^{\rm b} \mathscr{P}$. Two such morphisms $(g,D,f)$ and $(g',D',f')$ are equivalent if there exists a third such morphism $(g'',D'',f'')$, and two morphisms $h \colon D'' \to D$ and $h' \colon D'' \to D'$ in $\K^{-, \rm b} \mathscr{P}$, such that the diagram
\begin{center}
  \begin{tikzpicture}[text centered]
    % Nodes
    \node (C1) at (-0.5,0){$C$};
    \node (C2) at (2.5,0){$C'$};
    \node (D) at (1,1.2){$D$};
    \node (D') at (1,-1.2){$D'$};
    \node (D'') at (1,0){$D''$};
    
    \begin{scope}[->,font=\scriptsize,midway]
      % Arrows
      \draw (D) -- node[above left=-2pt]{$g$} (C1);
      \draw (D) -- node[above right=-2pt]{$f$} (C2);
      \draw (D') -- node[below left=-2pt]{$g'$} (C1);
      \draw (D') -- node[below right=-2pt]{$f'$} (C2);
      \draw (D'') -- node[above]{$g''$} (C1);
      \draw (D'') -- node[above]{$f''$} (C2);
      \draw (D'') -- node[right]{$h$} (D);
      \draw (D'') -- node[right]{$h'$} (D');
      
    \end{scope}
  \end{tikzpicture}
\end{center}
is commutative. For further details, we refer to \cite[Chapter 2]{Neeman}. The natural triangle functor $\K^{-, \rm b} \mathscr{P} \to \K^{-, \rm b} \mathscr{P} / \K^{\rm b} \mathscr{P}$ maps an object to itself, and a morphism $f \colon C \to C'$ to the equivalence class of the diagram
\begin{center}
  \begin{tikzpicture}[text centered]
    % Nodes
    \node (C1) at (0,0){$C$};
    \node (C2) at (2,0){$C'$};
    \node (D) at (1,1){$C$};
    
    \begin{scope}[->,font=\scriptsize,midway]
      % Arrows
      \draw (D) -- node[above left=-2pt]{$1$} (C1);
      \draw (D) -- node[above right=-2pt]{$f$} (C2);
      
    \end{scope}
  \end{tikzpicture}
\end{center}
In Theorem \ref{truncation}, we establish a fully faithful triangle functor from $\Kac \mathscr{P}$ to the quotient $\K^{-, \rm b} \mathscr{P} / \K^{\rm b} \mathscr{P}$. To avoid too many technicalities in the proof, we first prove the following lemma. It allows us to complete certain morphisms and homotopies of truncated complexes. Given an integer $n$ and a  complex
$$
C \colon \quad \cdots \xrightarrow{d_{n+3}}  C_{n+2} \xrightarrow{d_{n+2}} C_{n+1} \xrightarrow{d_{n+1}} C_n \xrightarrow{d_{n}} C_{n-1} \xrightarrow{d_{n-1}} C_{n-2} \xrightarrow{d_{n-2}} \cdots
$$
in $\mathscr{P}$, we denote its \emph{brutal truncation} at degree $n$ by 
$$
\beta_{\ge n}(C)\colon \quad \cdots \xrightarrow{d_{n+3}}  C_{n+2} \xrightarrow{d_{n+2}} C_{n+1} \xrightarrow{d_{n+1}} C_n \xrightarrow{} 0 \xrightarrow{} 0 \xrightarrow{} \cdots
$$
Note that when $C \in \Kac \mathscr{P}$, then $\beta_{\ge n}(C) \in \K^{-, \rm b} \mathscr{P}$.

\begin{lemma}\label{extension}
\sloppy Let $\mathscr{P}$ be an a additive category, and $C,D$ two complexes in $ \mathscr{P}$ with $C$ totally acyclic. Furthermore, let $n$ be an integer, and $f \colon \beta_{\ge n}(C) \to \beta_{\ge n}(D)$ a chain map. 
\begin{enumerate}
\item[(a)] The map $f$ can be extended to a chain map $\widehat{f} \colon C \to D$ (with $\widehat{f}_i =f_i$ for all $i\ge n$). 
\item[(b)] Let $\pi \colon C \to \beta_{\ge n}(C)$ be the natural chain map, and consider the composite chain map $f \circ \pi \colon C \to \beta_{\ge n}(D)$. If $f \circ \pi$ is nullhomotopic through a homotopy $h$, then $h$ can be extended to a homotopy $\widehat{h}$ making $\widehat{f}$ nullhomotopic.
\end{enumerate}
\end{lemma}

\begin{proof}
(a) The chain map $f$ is given by the solid part of the commutative diagram
\[ \begin{tikzpicture}
 \matrix [row sep=1cm, column sep=1cm] {
  \node (a3) {$\cdots$}; & \node (a2) {$C_{n+2}$}; & \node (a1) {$C_{n+1}$}; & \node (a0) {$C_n$}; & \node (a-1) {$C_{n-1}$}; & \node (a-2) {$C_{n-2}$}; & \node (a-3) {$\cdots$};\\  
  \node (b3) {$\cdots$}; & \node (b2) {$D_{n+2}$}; & \node (b1) {$D_{n+1}$}; & \node (b0) {$D_n$}; & \node (b-1) {$D_{n-1}$}; & \node (b-2) {$D_{n-2}$}; & \node (b-3) {$\cdots$}; \\
 };
 \draw [->] (a2) -- node [right] {$\scriptstyle f_{n+2}$} (b2);
 \draw [->] (a1) -- node [right] {$\scriptstyle f_{n+1}$} (b1);
 \draw [->] (a0) -- node [right] {$\scriptstyle f_{n}$} (b0);
 \draw [->,dashed] (a-1) -- node [right] {$\scriptstyle f_{n-1}$} (b-1);
 \foreach \i/\j/\l in {3/2/+3,2/1/+2,1/0/+1,0/-1/,-1/-2/-1,-2/-3/-2}
  {
   \draw [->] (a\i) -- node [above] {$\scriptstyle d_{n\l}^C$} (a\j);
   \draw [->] (b\i) -- node [above] {$\scriptstyle d_{n\l}^D$} (b\j);
  };
\end{tikzpicture} \]
and it suffices to find a map $f_{n-1}$ as indicated, making the square to its left commutative.
The composition $d_n^D \circ f_n \circ d_{n+1}^C$ is zero, and by assumption the sequence
\[ \Hom_{\mathscr{P}}(C_{n-1}, D_{n-1}) \xrightarrow{(d_n^C)^*} \Hom_{\mathscr{P}}(C_{n}, D_{n-1}) \xrightarrow{(d_{n+1}^C)^*} \Hom_{\mathscr{P}}(C_{n+1}, D_{n-1}) \]
is exact. Therefore, there exists a map $f_{n-1} \colon C_{n-1} \to D_{n-1}$ such that $f_{n-1} \circ d_n^C = d_n^D \circ f_n$.

(b) The homotopy $h$ is given by maps $h_{n-1},h_n,h_{n+1}, \dots$ as in the diagram
\[ \begin{tikzpicture}
 \matrix [row sep=1cm, column sep=1cm] {
  \node (a3) {$\cdots$}; & \node (a2) {$C_{n+2}$}; & \node (a1) {$C_{n+1}$}; & \node (a0) {$C_n$}; & \node (a-1) {$C_{n-1}$}; & \node (a-2) {$C_{n-2}$}; & \node (a-3) {$\cdots$};\\  
  \node (b3) {$\cdots$}; & \node (b2) {$D_{n+2}$}; & \node (b1) {$D_{n+1}$}; & \node (b0) {$D_n$}; & \node (b-1) {$D_{n-1}$}; & \node (b-2) {$D_{n-2}$}; & \node (b-3) {$\cdots$}; \\
 };
 \draw [->] (a2) -- node [right] {$\scriptstyle f_{n+2}$} (b2);
 \draw [->] (a1) -- node [right] {$\scriptstyle f_{n+1}$} (b1);
 \draw [->] (a0) -- node [right] {$\scriptstyle f_{n}$} (b0);
 \draw [->] (a-1) -- node [right] {$\scriptstyle f_{n-1}$} (b-1);
 \draw [->] (a-2) -- node [right] {$\scriptstyle f_{n-2}$} (b-2);
 \draw [->,dashed] (a-3) -- node [right] {$\scriptstyle h_{n-3}$} (b-2);
 \draw [->,dashed] (a-2) -- node [right] {$\scriptstyle h_{n-2}$} (b-1);
 \draw [->] (a-1) -- node [right] {$\scriptstyle h_{n-1}$} (b0);
 \draw [->] (a0) -- node [right] {$\scriptstyle h_n$} (b1);
 \draw [->] (a1) -- node [right] {$\scriptstyle h_{n+1}$} (b2);
 \draw [->] (a2) -- node [right] {$\scriptstyle h_{n+2}$} (b3);
 \foreach \i/\j/\l in {3/2/+3,2/1/+2,1/0/+1,0/-1/,-1/-2/-1,-2/-3/-2}
  {
   \draw [->] (a\i) -- node [above] {$\scriptstyle d_{n\l}^C$} (a\j);
   \draw [->] (b\i) -- node [above] {$\scriptstyle d_{n\l}^D$} (b\j);
  };
\end{tikzpicture} \]
with $f_i = d^D_{i+1} \circ h_n + h_{i-1} \circ d^C_i$ for all $i \ge n$. We must find maps $h_{n-2},h_{n-3}, \dots$ completing the homotopy.

To find the map $h_{n-2}$, consider the map $f_{n-1} - d_n^D \circ h_{n-1}$ in $\Hom_{\mathscr{P}}(C_{n-1},D_{n-1})$, for which we obtain
$$(f_{n-1} - d_n^D \circ h_{n-1}) \circ d_n^C = d_n^D \circ f_n - d_n^D \circ ( f_n - d_{n+1}^D \circ h_n) =0.$$
Since the sequence
\[ \Hom_{\mathscr{P}}(C_{n-2}, D_{n-1}) \xrightarrow{(d_{n-1}^C)^*} \Hom_{\mathscr{P}}(C_{n-1}, D_{n-1}) \xrightarrow{(d_{n}^C)^*} \Hom_{\mathscr{P}}(C_{n}, D_{n-1}) \]
is exact, there exists a map $h_{n-2} \colon C_{n-2} \to D_{n-1}$ such that
$$f_{n-1} - d_n^D \circ h_{n-1} = h_{n-2} \circ d_{n-1}^C.$$
Iterating this procedure, we obtain the maps $h_{n-3},h_{n-4}, \dots$ giving $\widehat{h}$.
\end{proof}

\section{Properties of the functor}\label{functor}

We are now ready to prove Theorem \ref{truncation}. It establishes a fully faithful triangle functor from $\Kac \mathscr{P}$ to the quotient $\K^{-, \rm b} \mathscr{P} / \K^{\rm b} \mathscr{P}$, mapping a complex $C$ to its brutal truncation $\beta_{\ge 0}(C)$
at degree zero. Note that brutal truncation does not define a functor between homotopy categories.

\begin{theorem}\label{truncation}
For an additive category $\mathscr{P}$, brutal truncation at degree zero induces a fully faithful triangle functor
\[ \beta_{\mathscr{P}} \colon \Kac \mathscr{P} \to \K^{-, \rm b} \mathscr{P} / \K^{\rm b} \mathscr{P}. \]
\end{theorem}

\begin{proof}
To simplify notation, we denote $\beta_{\mathscr{P}}$ by just $\beta$. The first issue to address is well-definedness. Let $f \colon C \to D$ be a map of complexes in $\Kac \mathscr{P}$, as indicated in the following diagram:
\[ \begin{tikzpicture}
 \matrix [row sep=1cm, column sep=1.5cm] {
  \node (a3) {$\cdots$}; & \node (a2) {$C_2$}; & \node (a1) {$C_1$}; & \node (a0) {$C_0$}; & \node (a-1) {$C_{-1}$}; & \node (a-2) {$\cdots$}; \\
  \node (b3) {$\cdots$}; & \node (b2) {$D_2$}; & \node (b1) {$D_1$}; & \node (b0) {$D_0$}; & \node (b-1) {$D_{-1}$}; & \node (b-2) {$\cdots$}; \\
 };
 \foreach \i in {-1,...,2} \draw [->] (a\i) -- node [right] {$\scriptstyle f_{\i}$} (b\i);
 \foreach \i/\j in {3/2,2/1,1/0,0/-1,-1/-2}
  {
   \draw [->] (a\i) -- node [above] {$\scriptstyle d_{\i}^C$} (a\j);
   \draw [->] (b\i) -- node [above] {$\scriptstyle d_{\i}^D$} (b\j);
   \draw [dashed, ->] (a\j) -- node [below right=-2pt] {$\scriptstyle h_{\j}$} (b\i);
  };
\end{tikzpicture} \]
It suffices to show that if $f$ vanishes in $\Kac \mathscr{P}$, that is, if there is a homotopy $h$ as indicated by the dashed arrows above, then $\beta_{\ge 0} (f)$ vanishes in $\K^{-, \rm b} \mathscr{P} / \K^{\rm b} \mathscr{P}$. Using this homotopy $h$, we see that the map
\[ \begin{tikzpicture}
 \matrix [row sep=1cm, column sep=1.5cm] {
  \node (a3) {$\cdots$}; & \node (a2) {$C_2$}; & \node (a1) {$C_1$}; & \node (a0) {$C_0$}; & \node (a-1) {$0$}; & \node (a-2) {$\cdots$}; \\
  \node (b3) {$\cdots$}; & \node (b2) {$D_2$}; & \node (b1) {$D_1$}; & \node (b0) {$D_0$}; & \node (b-1) {$0$}; & \node (b-2) {$\cdots$}; \\
 };
 \foreach \i in {1,...,2} \draw [->] (a\i) -- node [right] {$\scriptstyle f_{\i}$} (b\i);
 \foreach \i/\j in {3/2,2/1,1/0}
  {
   \draw [->] (a\i) -- node [above] {$\scriptstyle d_{\i}^C$} (a\j);
   \draw [->] (b\i) -- node [above] {$\scriptstyle d_{\i}^D$} (b\j);
  };
  \draw [->] (a0) -- node [above] {$\scriptstyle$} (a-1);
   \draw [->] (a-1) -- node [above] {$\scriptstyle$} (a-2);
   \draw [->] (b0) -- node [above] {$\scriptstyle$} (b-1);
   \draw [->] (b-1) -- node [above] {$\scriptstyle$} (b-2);
   \draw [->] (a-1) -- node [above] {$\scriptstyle$} (b-1);
   \draw [->] (a0) -- node [right] {$\scriptstyle f_0-h_{-1} \circ d_0^C$} (b0);
\end{tikzpicture} \]
is nullhomotopic in $\K^{-, \rm b} \mathscr{P}$. Consequently, the map $\beta_{\ge 0} (f)$ is homotopic to the map
\[ \begin{tikzpicture}
 \matrix [row sep=1cm, column sep=1.5cm] {
  \node (a3) {$\cdots$}; & \node (a2) {$C_2$}; & \node (a1) {$C_1$}; & \node (a0) {$C_0$}; & \node (a-1) {$0$}; & \node (a-2) {$\cdots$}; \\
  \node (b3) {$\cdots$}; & \node (b2) {$D_2$}; & \node (b1) {$D_1$}; & \node (b0) {$D_0$}; & \node (b-1) {$0$}; & \node (b-2) {$\cdots$}; \\
 };
 \foreach \i/\j in {3/2,2/1,1/0}
  {
   \draw [->] (a\i) -- node [above] {$\scriptstyle d_{\i}^C$} (a\j);
   \draw [->] (b\i) -- node [above] {$\scriptstyle d_{\i}^D$} (b\j);
  };
  \draw [->] (a0) -- node [above] {} (a-1);
   \draw [->] (a-1) -- node [above] {} (a-2);
   \draw [->] (b0) -- node [above] {} (b-1);
   \draw [->] (b-1) -- node [above] {} (b-2);
   \draw [->] (a-1) -- node [above] {} (b-1);
   \draw [->] (a0) -- node [right] {$\scriptstyle h_{-1} \circ d_0^C$} (b0);
   \draw [->] (a1) -- node [right] {$\scriptstyle 0$} (b1);
   \draw [->] (a2) -- node [right] {$\scriptstyle 0$} (b2);
\end{tikzpicture} \]
in $\K^{-, \rm b} \mathscr{P}$. Clearly, this map factors through the stalk complex with $C_{-1}$ in degree zero. Therefore $\beta (f)$, which equals the image of $\beta_{\ge 0} (f)$ in the quotient $\K^{-, \rm b} \mathscr{P} / \K^{\rm b} \mathscr{P}$, vanishes by  \cite[Lemma 2.1.26]{Neeman}. This shows that the functor $\beta$ is well-defined.

It is easy to check that $\beta$ is a triangle functor. The natural isomorphism $\beta \circ \Sigma \to \text{} \Sigma \circ \beta$ is given by
\[ \begin{tikzpicture}
 \matrix [row sep=.15cm, column sep=1cm] {
  & & \node [rotate=45,inner sep=0pt, outer sep=0pt] {\scriptsize deg $2$}; & \node [rotate=45,inner sep=0pt, outer sep=0pt] {\scriptsize deg $1$}; & \node [rotate=45,inner sep=0pt, outer sep=0pt] {\scriptsize deg $0$}; & \node [rotate=45,inner sep=0pt, outer sep=0pt] {\scriptsize deg $-1$}; \\
  \node (A) {$\beta( \Sigma C) \colon$}; & \node (a3) {$\cdots$}; & \node (a2) {$C_1$}; & \node (a1) {$C_0$}; & \node (a0) {$C_{-1}$}; & \node (a-1) {$0$}; & \node (a-2) {$\cdots$}; \\ \\ \\
  \node (B) {$\Sigma ( \beta (C)) \colon$}; & \node (b3) {$\cdots$}; & \node (b2) {$C_1$}; & \node (b1) {$C_0$}; & \node (b0) {$0$}; & \node (b-1) {$0$}; & \node (b-2) {$\cdots$}; \\
 };
 \draw [->] (A) -- (B);
 \foreach \i/\j in {3/2,2/1,1/0,0/-1,-1/-2}
  {
   \draw [->] (a\i) -- (a\j);
   \draw [->] (b\i) -- (b\j);
 };
 \draw [-,double] (a2) -- node [above] {} (b2);
 \draw [-,double] (a1) -- node [above] {} (b1);
 \draw [->] (a0) -- node [above] {} (b0);
 \draw [->] (a-1) -- node [above] {} (b-1);
\end{tikzpicture} \]
This is indeed an isomorphism in $\K^{-,\rm b} \mathscr{P} / \K^{\rm b} \mathscr{P}$, since its mapping cone in $\K^{-,\rm b} \mathscr{P}$ is isomorphic to the stalk complex with $C_{-1}$ in degree one, which belongs to $\K^{\rm b} \mathscr{P}$. Using a similar isomorphism, one checks that $\beta$ commutes with mapping cones. 

Next, we prove that $\beta$ is faithful. Let $f \colon C \to D$ be a morphism in $\Kac \mathscr{P}$ such that $\beta (f) = 0$. We may think of $f$ as a morphism of complexes. Then the condition $\beta (f) = 0$ means that, up to homotopy, the brutal truncation $\beta_{\geq 0} (f)$ factors through a bounded complex $C' \in \K^{\rm b} \mathscr{P}$. Choose a positive integer $n$ such that $C'_i =0$ for $i \ge n$. By truncating at degree $n$, we see that the induced map $\beta_{\geq n} (f) \circ \pi \colon C \to \beta_{\geq n} (D)$ of complexes is nullhomotopic. It then follows from Lemma \ref{extension}(b) that $f$ itself is nullhomotopic, and this shows that the functor $\beta$ is faithful.

It remains to show that $\beta$ is full. Let $\psi \colon \beta (C) \to \beta (D)$ be a morphism in $\K^{-, \rm b} \mathscr{P} / \K^{\rm b} \mathscr{P}$ between two complexes in the image of $\beta$. Then $\psi$ is represented by a diagram 
\begin{center}
  \begin{tikzpicture}[text centered]
    % Nodes
    \node (C) at (0,0){$\beta_{\ge 0} (C)$};
    \node (D) at (2,0){$\beta_{\ge 0} (D)$};
    \node (C') at (1,1){$C'$};
    
    \begin{scope}[->,font=\scriptsize,midway]
      % Arrows
      \draw (C') -- node[above left=-2pt]{$g$} (C);
      \draw (C') -- node[above right=-2pt]{$f$} (D);
      
    \end{scope}
  \end{tikzpicture}
\end{center}
of complexes and maps in $\K^{-, \rm b} \mathscr{P}$, where the mapping cone $C(g)$ of the map $g$ belongs to $\K^{\rm b} \mathscr{P}$. Up to homotopy, in sufficiently high degrees, the complex $C'$ then coincides with $\beta_{\ge 0} (C)$, and then also with $C$. Therefore, for some positive integer $n$, there is an equality $\beta_{\ge n}(C) = \beta_{\ge n}(C')$, and the truncation $\beta_{\ge n}(g)$ is the identity. Furthermore, the truncation $\beta_{\ge n}(f)$ is a morphism $\beta_{\ge n}(f) \colon \beta_{\ge n}(C) \to \beta_{\ge n}(D)$. By Lemma \ref{extension}(a), it admits and extension $\widehat{f} \colon C \to D$ of complexes in $\Kac \mathscr{P}$: we shall prove that $\psi = \beta ( \widehat{f} )$.

\sloppy Consider the solid part of the diagram
\begin{center}
  \begin{tikzpicture}[text centered]
    % Nodes
    \node (C) at (-2,0){$\beta_{\ge 0} (C)$};
    \node (D) at (4,0){$\beta_{\ge 0} (D)$};
    \node (C') at (1,1.5){$C'$};
    \node (C'') at (1,-1.5){$\beta_{\ge 0}(C)$};
    \node (C0) at (1,0){$\beta_{\ge -1}(C)$};
    
    \begin{scope}[->,font=\scriptsize,midway]
      % Arrows
      \draw (C') -- node[above left=-2pt]{$g$} (C);
      \draw (C') -- node[above right=-2pt]{$f$} (D);
      \draw (C'') -- node[below left=-2pt]{$1$} (C);
      \draw (C'') -- node[below right=-4pt]{$\beta_{\ge 0}( \widehat{f} )$} (D);
      \draw (C0) -- node[above]{$\pi$} (C);
      \draw (C0) -- node[above=-1pt]{$\beta_{\ge 0}( \widehat{f} ) \circ \pi$} (D);
      \draw[dashed] (C0) -- node[right=-1pt]{$\theta$} (C');
      \draw (C0) -- node[right=-1pt]{$\pi$} (C'');
     
    \end{scope}
  \end{tikzpicture}
\end{center}
of complexes and maps in $\K^{-, \rm b} \mathscr{P}$, where $\pi$ is the natural projection. The lower two triangles obviously commute. Furthermore, by Lemma \ref{extension}(a), 
the identity chain map $1 \colon \beta_{\ge n} (C) \to \beta_{\ge n} (C')$ admits an extension $\theta \colon \beta_{\ge -1} (C) \to C'$, and by construction the equalities
\begin{eqnarray*}
\beta_{\ge n} ( \pi ) & = & \beta_{\ge n} ( g \circ \theta ) \\
\beta_{\ge n} \left (  \beta_{\ge 0}( \widehat{f} ) \circ \pi \right ) & = & \beta_{\ge n} ( f \circ \theta )
\end{eqnarray*}
hold. The chain map $\pi - g \circ \theta$ can be viewed as a chain map $C \to \beta_{\ge 0} (C)$, and its truncation $\beta_{\ge n} ( \pi - g \circ \theta )$ is trivially nullhomotopic. Thus, by Lemma \ref{extension}(b), the map $\pi - g \circ \theta$ itself is nullhomotopic, and this shows that the top left triangle commutes. Similarly, the top right triangle commutes, hence the diagram is commutative. Consequently, the map $\psi$ equals $\beta ( \widehat{f} )$, and so the functor $\beta$ is full.
\end{proof}

We shall apply Theorem \ref{truncation} to the case when the additive category $\mathscr{P}$ is the category $\proj A$ of finitely generated projective modules over a left Noetherian ring $A$. The following result shows that, in this situation, if the functor is dense (i.e.\ an equivalence in view of Theorem \ref{truncation}), then all higher extensions between any module and the ring vanish. Recall first that the Verdier quotient
$$\K^{-,\rm b}(\proj A) / \K^{\rm b}(\proj A)$$
is the classical \emph{singularity category} $\D^{\rm b}_{\rm sg}(A)$ of $A$.

\begin{proposition}\label{dense}
\sloppy Let $A$ be a left Noetherian ring and $\proj A$ the category of finitely generated projective left $A$-modules. If the functor 
$$\beta_{\proj A} \colon \Kac(\proj A) \to \D^{\rm b}_{\rm sg}(A)$$
is dense, then $\Ext_A^n(M,A)=0$ for all $n \gg 0$ and every finitely generated left $A$-module $M$.
\end{proposition}

\begin{proof}
Let $M$ be a finitely generated left $A$-module, and
\[ P \colon \cdots \xrightarrow{d_4} P_3 \xrightarrow{d_3} P_2 \xrightarrow{d_2} P_1 \xrightarrow{d_1} P_0  \]
its projective resolution: this is a complex in $\K^{-,\rm b}(\proj A)$. Since the functor $\beta_{\proj A}$ is dense, the complex $P$ is isomorphic in $\D^{\rm b}_{\rm sg}(A)$ to $\beta_{\proj A} (T)$ for some totally acyclic complex $T \in \Kac(\proj A)$. For some $n \gg 0$, the two complexes $\beta_{\ge n}(P)$ and $\beta_{\ge n}(T)$ coincide up to homotopy, giving
$$
\Ext_A^i(M,A) \cong \Ho_{-i} \left ( \Hom_A(P,A) \right ) \cong \Ho_{-i} \left ( \Hom_A(T,A) \right )
$$
for all $i \ge n+1$. Since the complex $T$ is totally acyclic, the group $\Ho_{j} \left ( \Hom_A(T,A) \right )$ vanishes for all $j \in \mathbb{Z}$, and this proves the result.
\end{proof}

Specializing to the case when the ring is either left Artin or commutative Noetherian local, we obtain the following two corollaries. Recall that a commutative local ring is \emph{Gorenstein} if its injective dimension (as a module over itself) is finite.

\begin{corollary}\label{cor:dense-Artin}
Let $A$ be a left Artin ring, and $\proj A$ the category of finitely generated projective left $A$-modules. If the functor 
$$\beta_{\proj A} \colon \Kac(\proj A) \to \D^{\rm b}_{\rm sg}(A)$$
is dense, then the injective dimension of $A$ as a left module is finite.
\end{corollary}

\begin{proof}
There are finitely many simple left $A$-modules $S_1, \dots, S_t$, and by Proposition~\ref{dense} there exists an integer $n$ such that $\Ext_A^i( \oplus S_j,A)=0$ for $i \ge n+1$. The injective dimension of $A$ is therefore at most $n$. 
\end{proof}

\begin{corollary}\label{cor:dense-local}
Let $A$ be a commutative Noetherian local ring, and $\proj A$ the category of finitely generated projective (i.e.\ free) $A$-modules. If the functor 
$$\beta_{\proj A} \colon \Kac(\proj A) \to \D^{\rm b}_{\rm sg}(A)$$
is dense, then $A$ is Gorenstein.
\end{corollary}

\begin{proof}
Let $k$ be the residue field of $A$. By Proposition \ref{dense} there exists an integer $n$ such that $\Ext_A^i( k,A)=0$ for $i \ge n+1$. The injective dimension of $A$ is therefore at most $n$. 
\end{proof}

The following result shows that, in the situation of Proposition \ref{dense}, every injective module has finite projective dimension.

\begin{proposition}\label{injectives}
\sloppy Let $A$ be a left Noetherian ring and $\proj A$ the category of finitely generated projective left $A$-modules. If the functor 
$$\beta_{\proj A} \colon \Kac(\proj A) \to \D^{\rm b}_{\rm sg}(A)$$
is dense, then the projective dimension of every finitely generated injective left $A$-module is finite.
\end{proposition}

\begin{proof}
Let $I$ be a finitely generated injective left $A$-module, and $P_I \in \K^{-, \rm b}(\proj A)$ a projective resolution of $I$. For every $n \ge 1$, denote by $\Omega_A^n(I)$ the image of the $n$th differential in $P_I$. It suffices to show that the identity on $P_I$ factors through an object in $\K^{\rm b}(\proj A)$, for this would imply that $\Omega_A^n(I)$ is projective for high $n$.

Let $T$ be a totally acyclic complex in $\Kac(\proj A)$, and $M$ the image of its zeroth differential. Then there is a monomorphism $f \colon M \to P$ for some $P \in \proj A$ (take for example $P = T_{-1}$). Since $I$ is injective, every map $M \to I$ factors through $f$. Now for every $n \ge 1$, denote by $\Omega_A^n(M)$ the image of the $n$th differential in $T$. We claim that every map $g \colon \Omega_A^n(M) \to \Omega_A^n (I)$ factors through a projective module. To see this, note that every such map lifts to a chain map $\beta_{\ge n}(T) \to \beta_{\ge n}(P_I)$, and by Lemma \ref{extension}(a) this chain map can be extended to a chain map $T \to P_I$. This gives a map $g' \colon M \to I$, which factors through a projective module by the above. Since $g = \Omega_A^n(g')$, the map $g$ also factors through a projective module, as claimed.

We show next that $\Hom_{\D^{\rm b}_{\rm sg}(A)}( \beta_{\proj A}(T),P_I)=0$. Any morphism $\psi$ in this group is (represented by) a diagram
\begin{center}
  \begin{tikzpicture}[text centered]
    % Nodes
    \node (C1) at (0,0){$\beta_{\ge 0}(T)$};
    \node (C2) at (2,0){$P_I$};
    \node (D) at (1,1){$C$};
    
    \begin{scope}[->,font=\scriptsize,midway]
      % Arrows
      \draw (D) -- node[above left=-2pt]{$g$} (C1);
      \draw (D) -- node[above right=-2pt]{$f$} (C2);
      
    \end{scope}
  \end{tikzpicture}
\end{center}
in $\K^{-, \rm b}(\proj A)$, with the cone of $g$ belonging to $\K^{\rm b}(\proj A)$. As in the proof of Theorem \ref{truncation}, we can assume that (up to homotopy) the two complexes $\beta_{\ge 0}(T)$ and $C$ agree in high degrees. From the above, it then follows that for high $n$, every map $B^n(C) \to B^n (P_I)$ factors through a projective module, where $B^n(D)$ denotes the image of the $n$th differential in a complex $D$. This shows that $\psi =0$.

We can now show that the identity on $P_I$ factors through an object in $\K^{\rm b}(\proj A)$. Namely, since the functor $\beta_{\proj A}$ is dense, the complex $P_I$ is isomorphic in $\D^{\rm b}_{\rm sg}(A)$ to $\beta_{\proj A}(T)$ for some acyclic complex $T \in \Kac(\proj A)$. From what we showed above, we obtain
$$\Hom_{\D^{\rm b}_{\rm sg}(A)}( P_I,P_I) \simeq \Hom_{\D^{\rm b}_{\rm sg}(A)}( \beta_{\proj A}(T),P_I)=0,$$
and the result follows.
\end{proof}

Recall that a Noetherian ring (i.e.\ a ring that is both left and right Noetherian) is \emph{Gorenstein} if its injective dimensions both as a left and as a right module are finite. By a classical result of Zaks (cf.\ \cite[Lemma A]{Zaks}), the two injective dimensions then coincide. However, it is an open question whether a Noetherian ring of finite selfinjective dimension on one side is of finite selfinjective dimension on both sides, and therefore Gorenstein.

We have now come to the main result of this section. It deals with Artin rings (i.e.\ rings that are both left and right Artin) and commutative Noetherian local rings. Namely, for such a ring $A$, the functor $\beta_{\proj A}$ is dense if and only if $A$ is Gorenstein. As mentioned in the introduction, one may regard this as a triangulated category reformulation of the result of Auslander and Bridger 
\cite[Theorem (4.20)]{AuslanderBridger}; our approach was inspired by \cite[Theorem 4.4.1]{Buchweitz}.
   
\begin{theorem}\label{Gorenstein}
Let $A$ be either an Artin ring or a commutative Noetherian local ring, and $\proj A$ the category of finitely generated projective left $A$-modules. Then the functor 
$$\beta_{\proj A} \colon \Kac(\proj A) \to \D^{\rm b}_{\rm sg}(A)$$
is dense if and only if $A$ is Gorenstein.
\end{theorem}

\begin{proof}
Suppose the functor $\beta_{\proj A}$ is dense. If $A$ is local, then it is Gorenstein by Corollary \ref{cor:dense-local}. If $A$ is Artin, then the injective dimension of $A$ as a left module is finite by Corollary \ref{cor:dense-Artin}. Moreover, by Proposition \ref{injectives}, every finitely generated injective left $A$-module has finite projective dimension. The duality between finitely generated left and right modules then implies that every finitely generated projective right $A$-module has finite injective dimension. Therefore $A$ is Gorenstein.

Conversely, suppose that $A$ is Gorenstein, and let $C$ be a complex in $\K^{-,\rm b}(\proj A)$. Using the same notation as in the previous proof, there is an integer $n$ such that the $A$-module $B^n(C)$ is maximal Cohen-Macaulay, and such that $\beta_{\ge n}(C)$ is a projective resolution of $B^n(C)$. Since $B^n(C)$ is maximal Cohen-Macaulay, it admits a projective co-resolution $C'$, and splicing $\beta_{\ge n}(C)$ and $C'$ at $B^n(C)$ gives a totally acyclic complex $T \in \Kac(\proj A)$. Since $\beta_{\ge n}(C) = \beta_{\ge n}(T)$, the complexes $C$ and $\beta_{\proj A} (T)$ are isomorphic in $\D^{\rm b}_{\rm sg}(A)$, hence the functor $\beta_{\proj A}$ is dense.
\end{proof}

\section{Defect categories}\label{defect}

Motivated by Theorem \ref{truncation} and Theorem \ref{Gorenstein}, we introduce triangulated defect categories for any left Noetherian ring $A$. 

Since the triangle functor
$$\beta_{\proj A} \colon \Kac(\proj A) \to \D^{\rm b}_{\rm sg}(A)$$
is fully faithful by Theorem \ref{truncation}, the category $\Kac(\proj A)$ embeds in $\D^{\rm b}_{\rm sg}(A)$ as the image of $\beta_{\proj A}$. The isomorphism closure $\langle \Image \beta_{\proj A} \rangle$ is a thick subcategory of $\D^{\rm b}_{\rm sg}(A)$, hence we may form the corresponding Verdier quotient.

\begin{definition}
The \emph{Gorenstein defect category} of a left Noetherian ring $A$ is the Verdier quotient
$$\D^{\rm b}_{\rm G}(A) \stackrel{\text{def}}{=} \D^{\rm b}_{\rm sg}(A) / \langle \Image\beta_{\proj A} \rangle,$$
where $\proj A$ is the category of finitely generated projective left $A$-modules.
\end{definition}

In terms of the Gorenstein defect category, Theorem \ref{Gorenstein} takes the following form.

\begin{theorem}\label{GorensteinDefect}
If $A$ is either an Artin ring or a commutative Noetherian local ring, then $\D^{\rm b}_{\rm G}(A)=0$ if and only if $A$ is Gorenstein.
\end{theorem}

We can in fact define related defect categories as follows.  Let $\mathcal C$ be a thick subcategory of $\Kac(\proj A)$, and $\beta_{\mathcal C}$ the restriction of $\beta_{\proj A}$ to $\mathcal C$.  Then the category $\mathcal C$ embeds in $\D^{\rm b}_{\rm sg}(A)$ as the image of $\beta_{\mathcal C}$. The isomorphism closure $\langle \Image \beta_{\mathcal C} \rangle$ is a thick subcategory of $\D^{\rm b}_{\rm sg}(A)$, and we make the following definition.

\begin{definition}
The \emph{$\mathcal C$ defect category} of a left Noetherian ring $A$ is the Verdier quotient
$$
\D^{\rm b}_{\mathcal C}(A) \stackrel{\text{def}}{=} \D^{\rm b}_{\rm sg}(A) / \langle \Image\beta_{\mathcal C} \rangle
$$
\end{definition}

We have the following results complementing Theorem \ref{GorensteinDefect}, in the commutative local case. Recall that a complex $C$ over a commutative local ring $(A,\mathfrak m)$ is called \emph{minimal} if $d_n^C(C_n)\subseteq \mathfrak mC_{n-1}$ for all $n\in \mathbb Z$.  Every complex in $\Kac(\proj A)$ is isomorphic in $\Kac(\proj A)$ to a minimal complex which is unique up to isomorphism. We define the \emph{complexity} of a minimal complex $C\in\Kac(\proj A)$ to be 
$$
\cx_AC \stackrel{\text{def}}{=} \inf\{d\in\mathbb N \mid \rank(C_n\oplus C_{-n}) \le bn^{d-1}\text{ for } n \gg 0 \text{ and some real number $b$}\},
$$
and extend this definition to all of $\Kac(\proj A)$ by setting $\cx_AC =\cx_AC'$ if $C$ is isomorphic in $\Kac(\proj A)$ to a minimal complex $C'$.

\begin{theorem} Let $A$ be a commutative local ring.

(1) Let $\mathcal C$ be the full subcategory of $\Kac(\proj A)$ consisting of complexes of finite complexity.  Then $\D^{\rm b}_{\mathcal C}(A)=0$ if an only if $A$ is a complete intersection.

(2) Let $\mathcal C$ be the full subcategory of $\Kac(\proj A)$ consisting of complexes of complexity at most $c$.  Then $\D^{\rm b}_{\mathcal C}(A)=0$ if an only if $A$ is a codimension $c$ complete intersection.
\end{theorem}

\begin{proof}
For (1), if $\D^{\rm b}_{\mathcal C}(A)=0$ then the free resolution of the residue field of $A$ has finite complexity. By a classical result of Gulliksen \cite[Theorem 2.3]{Gulliksen}, this implies that $A$ is a complete intersection. Conversely, the same result by Gulliksen shows that if $A$ is a complete intersection, then every minimal free resolution has finite complexity.

For (2), it is well known that over a complete intersection, the complexity of a module is at most that of the residue field, and the latter equals the codimension; see for instance \cite[Remark 8.1.1.2]{Avramov}.
\end{proof}

Theorem \ref{GorensteinDefect} suggests that the size of the Gorenstein defect category (of an Artin ring or a commutative Noetherian local ring) measures in some sense ``how far" the ring is from being Gorenstein. One measure of the size of the Gorenstein defect category is its dimension in the sense of Rouquier (cf.\ \cite{Rouquier}); we recall now the definition. Let $\mathcal T$ be a triangulated category with suspension functor $\s$. Given two subcategories $\mathcal A$ and $\mathcal B$ of $\mathcal T$, we denote by
$\mathcal A \ast \mathcal B$ the following full subcategory of $\mathcal T$: an object $X \in \mathcal T$ belongs to $\mathcal A \ast \mathcal B$ if and only if there exists a distinguished triangle
$$A \to X \to B \to \s A$$
in $\mathcal T$, with $A \in \mathcal A$ and $B \in \mathcal B$. Furthermore, we denote by $\thick_{\mathcal T}^1 ( \mathcal A )$ the full subcategory of $\mathcal T$ consisting of all objects isomorphic to a direct summand of a finite direct sum of suspensions of objects in $\mathcal A$. Then we inductively define $\thick_{\mathcal T}^n ( \mathcal A )$ by$$\thick_{\mathcal T}^n ( \mathcal A ) \stackrel{\text{def}}{=} \thick_{\mathcal T}^1 \left ( \thick_{\mathcal T}^{n-1} ( \mathcal A ) \ast \thick_{\mathcal T}^1 ( \mathcal A ) \right ).$$
The \emph{dimension} of $\mathcal T$, denoted $\dim \mathcal T$, is defined as 
$$\dim \mathcal T \stackrel{\text{def}}{=} \inf \{ d \in \mathbb{Z} \mid \text{ there exists an object } A \in \mathcal T \text{ such that } \mathcal T = \thick_{\mathcal T}^{d+1}(A) \}.$$
It would be interesting to find criteria which characterize the rings whose Gorenstein defect categories are $n$-dimensional. An answer to the following question would be a natural start.

\begin{question}
What characterizes Artin rings and commutative Noetherian local rings with zero-dimensional Gorenstein defect categories?
\end{question}

We end with an important example; it is the smallest commutative local ring --- in a sense which we make precise --- for which the Gorenstein defect category is interesting. We show that the dimension of the defect category is at most one. Note that it is also a finite dimensional $k$-algebra. 

We need the following lemma, which gives an upper bound on the dimension of the singularity category. It improves \cite[Proposition 3.7]{Rouquier1}; if $A$ is an Artin ring, then the lemma shows that the dimension of $ \D^{\rm b}_{\rm sg}(A)$ is at most $\ell \ell (A) -2$, where $\ell \ell (A)$ is the Loewy length of $A$.

\begin{lemma}\label{qube}
Let $A$ be either a commutative Noetherian local ring or an Artin ring, and let $\mathfrak r$ be its radical. If $\mathfrak r^n =0$, then $\dim \D^{\rm b}_{\rm sg}(A) \le n-2$.
\end{lemma}

\begin{proof}
Consider the stalk complex $A / \mathfrak r$ in the bounded derived category $\D^{\rm b}(A)$. It follows from the proof of \cite[Lemma 7.35]{Rouquier} that $\D^{\rm b}(A) = \thick_{ \D^{\rm b}(A)}^n( A / \mathfrak r )$. The proof also shows that if $M$ is an $A$-module with $\mathfrak r^tM=0$, then as a stalk complex $M$ belongs to $\thick_{ \D^{\rm b}(A)}^t( A / \mathfrak r )$. In the singularity category $\D^{\rm b}_{\rm sg}(A)$, every indecomposable object is isomorphic to a stalk complex of a module, and this stalk complex is again isomorphic to the shift of the stalk complex given by its first syzygy. Since $\mathfrak r^{n-1}\Omega_A^1(M)=0$ for every $A$-module $M$, we see that $\D^{\rm b}_{\rm sg}(A) = \thick_{ \D^{\rm b}_{\rm sg}(A)}^{n-1}( A / \mathfrak r )$. Consequently, the dimension of $\D^{\rm b}_{\rm sg}(A)$ is at most $n-2$.
\end{proof}

\begin{example}
Let $k$ be a field and $R=k[x,y,z]/(x^2,y^2,z^2,yz)$, a quotient of a polynomial ring.  Note that  $R$ is a finite dimensional commutative local $k$-algebra with maximal ideal $\mathfrak n=(x,y,z)$, embedding dimension 3, Loewy length 3 and length 6.

The image of $\beta_{\proj R}$ is nonzero, since $R$ admits, for example, the totally acyclic complex
$$
\cdots \xrightarrow{x} R \xrightarrow{x} R \xrightarrow{x} R \xrightarrow{x}  \cdots
$$ 
However, $R$ is not Gorenstein as its socle is clearly 2-dimensional.  We claim that $R$ is the smallest non-Gorenstein ring to admit non-trivial totally acyclic complexes, in terms of the invariants listed above.  Indeed, consider a commutative local ring $A$ with maximal ideal $\mathfrak m$ satisfying $\mathfrak m^n=0$.  If the Loewy length of $A$ were 2, that is $n=2$,  then $A$ would be of minimal multiplicity, and non-Gorenstein rings with minimal multiplicity are Golod 
\cite[5.2.8]{Avramov}.  By \cite[3.5]{AvramovMartsinkovsky} Golod rings do not admit non-trivial totally reflexive modules. Thus one is forced to consider rings with Loewy length at least 3.  Now consider the embedding dimension of $A$.  If $A$ has embedding dimension 1 or 2, then $A$ is a complete intersection \cite{Scheja}, which is a Gorenstein ring. Therefore one must consider rings of embedding dimension at least 3.  Now by a theorem due to Christensen and Veliche \cite[Theorem A]{ChristensenVeliche}, if $A$ admits a totally acyclic complex, and $\mathfrak m^3=0$, then the socle of $A$ is $\mathfrak m^2$ and the length of $A$ is twice its embedding dimension.  Thus the ring $R$ above is the least commutative local non-Gorenstein ring, in terms of Loewy length, embedding dimension and length, which admits non-trivial totally acyclic complexes.

Since the maximal ideal $\mathfrak n$ of $R$ satisfies $\mathfrak n^3=0$, Lemma \ref{qube} shows that the dimension of the singularity category  $\D^{\rm b}_{\rm sg}(R)$ is at most one. Since the Gorenstein defect category $\D^{\rm b}_{\rm G}(A)$ is a quotient of $\D^{\rm b}_{\rm sg}(R)$, it follows from \cite[Lemma 3.4]{Rouquier} that $\dim \D^{\rm b}_{\rm G}(A) \le 1$.
\end{example}

\end{document}